\documentclass{article}
\usepackage{amsmath}
\usepackage{amssymb}
\usepackage{latexsym}
\usepackage{amsthm}
\usepackage{mathrsfs}
\usepackage{wasysym}
\usepackage{fancyhdr}
\usepackage{amsfonts}
\usepackage{amssymb}
\usepackage{epsfig}

\newtheorem{thm}{Theorem}[section]
\newtheorem{cor}[thm]{Corollary}
\newtheorem{lemma}[thm]{Lemma}
\newtheorem{prop}[thm]{Proposition}

\begin{document}

\title{Graph Invertibility and Median Eigenvalues}
  \author{Dong Ye\thanks{Department of Mathematical Sciences and Center for Computational Sciences, Middle Tennessee State University,  Murfreesboro, TN 37132; Email: dong.ye@mtsu.edu}\ ,  
Yujun Yang\thanks{Department of Marine Sciences, Texas A\&M University at Galveston, Galveston, TX 77553; Email: yangy@tamug.edu}\ ,  
Bholanath Mandal\thanks{Department of Marine Sciences, Texas A\&M University at Galveston, Galveston, TX 77553} 
\ and 
Douglas J. Klein\thanks{Department of Marine Sciences, Texas A\&M University at Galveston, Galveston, TX 77553; Email: kleind@tamug.edu}  
}
\date{}

\maketitle

\begin{abstract}
Let $(G,w)$ be a weighted graph with a weight-function $w: E(G)\to \mathbb R\backslash\{0\}$.
A weighted graph $(G,w)$ is invertible to a new weighted graph if its adjacency matrix is invertible. 
A graph inverse has combinatorial interest and can be applied to bound median eigenvalues of a graph 
such as have physical meanings in Quatumn Chemistry. 
In this paper, we characterize the inverse of a weighted graph based on its Sachs subgraphs
that are spanning subgraphs with only $K_2$ or cycles (or loops) as components. The characterization 
can be used to find the inverse of a weighted graph based on its structures instead of its adjacency matrix. 
If a graph has its spectra split about the origin, i.e., half of eigenvalues are positive and half of them are negative, 
then its median eigenvalues can be bounded by estimating the largest and smallest eigenvalues of its inverse.
We characterize graphs with a unique Sachs subgraph and prove that these graphs
has their spectra split about the origin if they have a perfect matching.
As applications, we show that the median eigenvalues of stellated graphs of trees
and corona graphs belong to different halves of the interval $[-1,1]$.
\end{abstract}

\section{Introduction} 

In this paper, graphs may contain loops but no multiple edges. Let $G$ be a graph with 
vertex set $V(G)$ and edge set $E(G)$. Its adjacency matrix  $\mathbb A$ is defined as
the $ij$-entry $(\mathbb A)_{ij}=1$ if $ij\in E(G)$ and $(\mathbb A)_{ij}=0$ otherwise.
Assume that $\lambda_1\ge \lambda_2\ge \cdots\ge \lambda_n$ where $n=|V(G)|$, are the 
eigenvalues of $\mathbb A$ (also the eigenvalues of $G$).
In Quantum Chemistry, 
the eigenvalues of a molecular graph  have physical meanings. For example, the sum of
absolute value of eigenvalues of a graph $G$, also called the {\em energy} of $G$ \cite{JK}, is 
often equal to the total H\"{u}ckel $\pi$-electron
energy of the molecule represented by $G$. Also many physico-chemical parameters of molecules are determined by or are dependent upon the HOMO-LUMO gap \cite{FP,GR}, which is often given as  the difference between the median eigenvalues, $\lambda_{H}-\lambda_L$, where $H=\lfloor (n+1)/2\rfloor$ and $L=\lceil (n+1)/2\rceil$.  

Throughout the more traditional chemical literature, there has been extensive effort to deal with the HOMO-LUMO gap mostly in an explicit consideration of individual molecules case by case. A few mathematical methods have been developed especially in the last decade to characterize the HOMO-LUMO gaps of graphs in a general manner \cite{G,K,KM, KYY,BM1,BM2,BM3,ZC}. Recently, Mohar introduced a graph partition method to bound
the HOMO-LUMO gaps for subcubic graphs \cite{BM1,BM2,BM3}. However, not all graphs have these nice partition properties
and a desired partition is hard to find, even for plane subcubic graphs \cite{BM1}.

It is well-known that the eigenvalues of
a bipartite graph are symmetric about the origin. If the adjacency matrix $\mathbb A$ of 
a bipartite graph $G$ is invertible, then the reciprocal of the maximum eigenvalue of $\mathbb A^{-1}$ is equal to the
$\lambda_H$ of $G$ and the reciprocal of the least eigenvalue of $\mathbb A^{-1}$ is equal to $\lambda_L$.
Based on this fact, the invertibility of adjacency matrices of trees had been discussed in order to evaluate their HOMO-LUMO gaps \cite{G,K}, and 
later the method has been extended to bipartite graphs with a unique perfect matching \cite{KM,SC}.
Besides the chemical interests, the invertibility of adjacency matrices of graphs is of independent interest as indicated in \cite{G, MM}. For examples, 
the invertibility of adjacency matrices of graphs has connections to other interesting combinatorial
topics such as M\"{o}bius inversion of partially ordered sets (see the treatment in
Chapter 2 of Lov\'asz \cite{L}) \cite{G,SC} and Motzkin numbers \cite{DS,MM}. 

Here, the aim of this paper
is to extend this idea to graphs with more general settings. Note that, the eigenvalues of non-bipartite graphs are
not symmetric about the origin. But, our methodology works when the eigenvalues of a graph evenly split about the 
origin, i.e., half of them are positive and half of them are negative. Another purpose of this paper is
to discuss the invertibility of graphs. It is very clear when a matrix is invertible. But the inverse of an
adjacency matrix of a graph is not necessarily an adjacency matrix of another graph. In fact, 
if the adjacency matrix $\mathbb A$ of a graph $G$ is invertible and    
$\mathbb A^{-1}$ is an adjacency matrix of another simple graph, then $G$ has to be the graph $nK_2$ \cite{HM}.
Godsil \cite{G} defined an inverse of a bipartite graph $G$ with a unique perfect matching to be a graph, denoted by $G^{-1}$ with adjacency matrix diagonally similar
to the inverse of adjacency matrix of $G$, i.e., the adjacency matrix of $G^{-1}$ is $\mathbb D\mathbb A^{-1} \mathbb D$ for some diagonal matrix $\mathbb D$ with entries $1$ or $-1$ on its diagonal, where $\mathbb A$ is the adjacency matrix
of $G$ (see also \cite{KM}). This definition uniquely defines the inverse of a graph. But the invertible graphs are 
still quite limited. In order to make a more encompassing definition of the {\em inverse} of a graph, McLeman and McNicholas \cite{MM} defined
that a graph $G_1$ is an inverse of a graph $G_2$ if $\lambda$ is an eigenvalue of $G_1$ if and only if $1/\lambda$ is an eigenvalue of $G_2$. But the inverses of a graph defined by McLeman and McNicholas are not always unique because there exist cospectral invertible graphs (see Figure~\ref{fig:co-inverse}, cf. \cite{GM}). Based on these two different definitions, the invertibility of bipartite graphs with unique 
perfect matching has been discussed in \cite{G, KM,MM,SC}.  However, their methods can not be easily used for non-bipartite graphs. In these papers, the authors seek to deal with signs in the computation of the inverse of the adjacency 
matrix of the original graph. In order to avoid the special treatment of signs, 
we here modify the definition of graph inverse to weighted graphs, with signs allowed on the weights. 
The modified definition can uniquely and broadly define the inverse for graphs. 

\begin{figure}[!hbtp]\refstepcounter{figure}\label{fig:co-inverse}
\begin{center}
\includegraphics[scale=1]{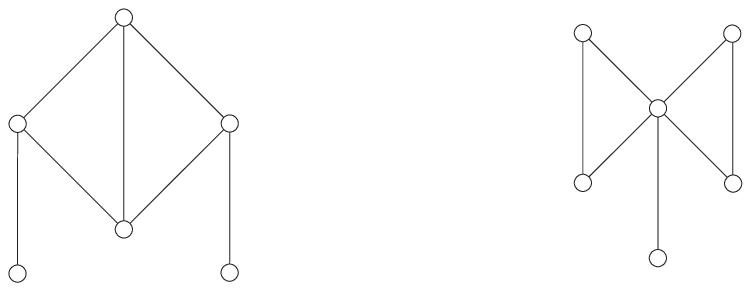}\\
{Figure \ref{fig:co-inverse}: The smallest cospectral invertible graphs.}
\end{center}
\end{figure}

A weighted graph $(G,w)$ is a graph with a weight function
$w: E(G)\to \mathbb R\backslash\{0\}$. The adjacency matrix of a weighted graph, denoted
by $\mathbb A$, is defined as
\[
\mathbb (\mathbb A)_{ij} :=\left\{
 \begin{array}{ll}
w(ij)    &\mbox{if $i j\in E(G)$;}\\
0  &\mbox{otherwise}
 \end{array}
 \right.
\]
where loops, with $w(ii)\ne 0$, are allowed. A weighted
graph $(G,w)$ is {\em invertible} if its adjacency matrix has an inverse that is also an adjacency matrix
of a weighted graph.  

A graph can be treated as a weighted graph with the constant weight function $w: E(G)\to 1$.
Another important family of weighted graphs is signed graphs. 
A {\em signed graph} $(G,\sigma)$ is a weighted graph  with a weight function 
$\sigma: E(G)\to \{-1,+1\}$, where $\sigma$ is called the {\em signature} of $G$ (see \cite{Z93}). Two 
signatures of a graph $G$ are {\em equivalent} to each other if one can be obtained from the other 
by changing the signatures of all edges in an edge-cut of $G$. 
A signed graph is {\em balanced} if it is equivalent to a graph.  In Godsil's definition, a bipartite graph
 with a unique perfect matching is {\em invertible} if its inverse is a balanced signed graph.
Signed graphs have extensive applications in combinatorics and matroid theory \cite{Z13}. For more details and 
interesting problems 
on signed graphs, one may refer to the survey of Zaslavsky \cite{Z98}. 

A {\em Sachs subgraph} is a spanning subgraph with only $K_2$ or cycles (including loops) as components. In this paper, we investigate the invertibility of weighted graphs and characterize their inverses
based on Sachs subgraphs. A characterization of graphs with a unique Sachs subgraph is 
obtained. For signed graphs with a unique Sachs subgraph, we show 
that their spectra split about the origin if they
have a perfect matching. As applications, we show that 
 the median eigenvalues of  stellated graphs of trees and corona graphs 
belong to the interval $[-1,1]$.

\section{Inverses of weighted graphs}

Let $(G,w)$ be a weighted graph. The following is a 
straight-forward proposition that establishes the 
equivalent relation between  the invertibility
of weighted graphs and the invertibility of real symmetric matrices.

\begin{prop}\label{prop-1}
Let $(G,w)$ be a weighted graph with adjacency matrix $\mathbb A$. Then $(G,w)$ is invertible if and only if $\det(\mathbb A)\ne 0$, and the inverse of an invertible weighted graph $(G,w)$ is unique.
\end{prop}
\begin{proof}
Let $(G,w)$ be a weighted graph. Then $\det(\mathbb A)\ne 0$ if and only if
then $\mathbb A$ has an inverse $\mathbb A^{-1}$. 
Note that $(\mathbb A^{-1})^{\intercal}=(\mathbb A^{\intercal})^{-1} =\mathbb A^{-1}$. Hence $\mathbb A^{-1}$ is a symmetric matrix, which is 
corresponding to a weighted graph. Note that, $\mathbb A^{-1}$ is unique. So the 
the inverse of $(G,w)$ is unique. The proposition follows.
\end{proof}

The adjacency matrix $\mathbb A$ of a weighted graph $(G,w)$ is a linear transformation from
the vector space $\mathbb R^n$
to itself. If $(G,w)$ is invertible, then $\mathbb A$ is full rank. In other words, all column vectors
$\mathbf b_i$ of $\mathbb A$ for $i\in V(G)$ form a basis of $\mathbb R^n$. 

Throughout the paper, we always use $(\mathbb A)_{ij}$ to denote the $(i,j)$-entry of the matrix $\mathbb A$, 
and $\mathbb A_{(ij)}$ to denote the submatrix of $\mathbb A$ by deleting the $i$-th row and 
the $j$-th column.

A subgraph $H$ of a graph $G$ is spanning if $V(H)=V(G)$. A weighted graph $(H, w_H)$ is a  subgraph of
$(G, w)$, if $H$ is a subgraph of $G$ and $w_H(e)=w(e)$ for any $e\in E(H)\subseteq E(G)$. The weight of a 
subgraph $H$ is defined as $w(H)=\prod_{e\in E(H)}w(e)$. A spanning subgraph $S$
is called a {\em Sachs subgraph} of $G$ if every component of $S$ is either $K_2$ or a cycle (including loops). 
For a Sachs subgraph $S$, denote the set of all cycles of $S$ by $\mathcal C $, the set of all loops of $S$ by $L$ and the set of all
$K_2$'s of $S$ by $M $ which is a matching of $G$. So a Sachs subgraph consists of three parts:
cycles, loops and a matching. In the rest of the paper, we also denote a Sachs subgraph as $S=\mathcal C\cup M\cup L$.
The determinant of adjacency matrix of a graph can be represented by its Sachs subgraphs as follows.

\begin{thm}[Harary,\cite{H}]\label{thm:Harary}
Let $G$ be a simple graph and $\mathbb A$ the adjacency matrix. Then
\[\det(\mathbb A)=\sum_{S}  2^{|\mathcal C|}(-1)^{|\mathcal C|+|E(S)|},\]
where $S=\mathcal C\cup M$ is a Sachs subgraph.
\end{thm} 

The following result extends Theorem~\ref{thm:Harary}  to weighted graphs $(G,w)$ which may contain loops.

\begin{thm}\label{thm:weight-det}
Let $(G,w)$ be a weighted graph and $\mathbb A$ the adjacency matrix.
Then
\[\det(\mathbb A)=\sum_{S} 2^{|\mathcal C|} w(\mathcal C\cup L)\ w^2(M) (-1)^{|\mathcal C|+|L|+|E(S)|},\]
where $S=\mathcal C\cup M\cup L$ is a Sachs subgraph.
\end{thm}
\begin{proof}
By the definition of determinant,
\[\det(\mathbb A)=\sum_{\pi} \mbox{sgn}(\pi)\prod_{i\in V(G)} (\mathbb A)_{i\pi(i)},\]
where $\pi$ is a permutation on $V(G)=\{1,2,..,n\}$. A permutation $\pi$ on $V(G)$ contributing
to $\det(\mathbb A)$ corresponds to a Sachs subgraph $S$: a cycle of $\pi$ with length $k\notin\{1,2\}$ corresponds to a cycle of $G$
with length $k$, a cycle of $\pi$ of length $1$ (fixing a vertex) corresponds to a loop of $G$, and a cycle
of length 2 corresponds to an edge (or $K_2$). 
Hence,
\begin{align*}\mbox{sgn}(\pi)\prod_{i\in V(G)} (\mathbb A)_{i\pi(i)}&=(-1)^{|\mathcal C|+|E(\mathcal C)|+|M|}\prod_{C\in \mathcal C}w(C) \cdot \prod_{(i)\in \pi} (\mathbb A)_{ii} \prod_{(ij)\in \pi} (\mathbb A)_{ij} (\mathbb A)_{ji}\\
&=(-1)^{|\mathcal C|+|L|+|E(S)|}w(\mathcal C)w(L)w^2(M).
\end{align*} 
But a Sachs subgraph $S=\mathcal C\cup M\cup L$ 
 corresponds to $2^{|\mathcal C|}$ permutations because for each cycle $C=i_1i_2\cdots i_k$ ($k\ge 3$), 
there are two different corresponding cyclic permutations $(i_1i_2\cdots i_k)$ and $(i_1i_ki_{k-1}\cdots i_3 i_2)$. 
So, 
\begin{align*}
\det(\mathbb A)&=\sum_{\pi} \mbox{sgn}(\pi)\prod_{i\in V(G)} (\mathbb A)_{i\pi(i)}\\
                         &=\sum_{S} 2^{|\mathcal C|} w(\mathcal C\cup L)\ w^2(M) (-1)^{|\mathcal C|+|L|+|E(S)|}.
\end{align*}
This completes the proof.
\end{proof}

For some special weighted graphs, the formula in Theorem~\ref{thm:weight-det} can be simplified. For example, via the following corollary.

\begin{cor}\label{cor-sign}
Let $(G,\sigma)$ be a signed simple graph with a unique Sachs subgraph.  If $G$ has a perfect matching $M$, then $\det(\mathbb A )=(-1)^{|M|}$,
where $\mathbb A$ is the adjacency matrix of $(G,\sigma)$.
\end{cor}

\begin{proof} Note that, a perfect matching is a Sachs subgraph. Since $(G,\sigma)$ is simple and
has a unique Sachs subgraph, it follows that the unique Sachs subgraph of $G$ is the perfect 
matching $M$.
By Theorem~\ref{thm:weight-det}, and the fact that  $\mathcal C=\emptyset$ and $L=\emptyset$,
\begin{align*}
\det(\mathbb A)&=\sum_{S}2^{|\mathcal C|}\sigma(\mathcal C\cup L)\sigma^2(M)(-1)^{|\mathcal C|+|L|+|E(S)|}\\
&=(-1)^{|M|}\sigma^2(M).
\end{align*}
Note that $\sigma:E(G)\to \{-1,1\}$. So $\sigma^2(M)=1$. Further, $\det(\mathbb A)=(-1)^{|M|}$.
\end{proof}

Let $(G,w)$ be an invertible weighted graph. The next result characterizes the inverse of $(G,w)$.

\begin{thm}\label{thm:weight-inverse}
Let $(G,w)$ be a weighted graph with adjacency matrix $\mathbb A$, and 
\[\mathcal P_{ij}=\{P| P \mbox{ is a path joining } i \mbox{ and } j\ne i \mbox{ such
that } G-V(P) \mbox{ has   a Sachs subgraph } S\}.\]
If $(G,w)$ has an inverse $(G^{-1}, w^{-1})$, then 
\[w^{-1}(ij)=\left \{
 \begin{array}{ll}
\displaystyle \frac{1}{\det (\mathbb A)} \sum_{P\in \mathcal P_{ij}} \Big (w(P) \big (\sum_{S}w(\mathcal C\cup L)\ w^2(M)\ 2^{|\mathcal C|}(-1)^{|\mathcal C|+|L|+|E(S)\cup E(P)|}\big )\Big) &\mbox{if } i\ne j; \\
\displaystyle \frac{1}{\det (\mathbb A)}\det(\mathbb A_{(ii)})&\mbox{otherwise}
\end{array} \right .\]
where $S=\mathcal C\cup M\cup L$ is a Sachs subgraph of $G-V(P)$. 
\end{thm}

\begin{proof}
Since $(G,w)$ is invertible,
$\mathbb A^{-1}$ exists and is an adjacency matrix of $(G^{-1},w^{-1})$. 
According to the definition of the inverse
of a weighted graph, $w^{-1}(ij)=(\mathbb A^{-1})_{ij}$. 

By Proposition~\ref{prop-1} and Cramer's rule,
\[(\mathbb A^{-1})_{ij}=(\mathbb A^{-1})_{ji}=\frac{ c_{ij} }{\det(\mathbb A)}\]
where $ c_{ij}=(-1)^{i+j} \det(\mathbb A_{(ij)})$.
Let $\mathbb M_{i,j}$ be the matrix obtained from $\mathbb A$ by replacing  the $(i,j)$-entry 
by 1 and all other entries 
in the $i$-th row and $j$-th column by 0. Then by Laplace expansion, $c_{ij}=\det(\mathbb M_{i,j})$

If $i=j$, then $\det(\mathbb M_{i,i})=\det(\mathbb A_{(ii)})$. So $w^{-1}(ii)=\det(\mathbb A_{(ii)})/\det(\mathbb A)$. So in the following, assume that $i\ne j$.

Since all $(i,k)$-entries ($k\ne j$) of $\mathbb M_{i,j}$ 
are equal to 0
and its $(i,j)$-entry is 1, only permutations taking $i$ to $j$ contribute 
to the the determinant of $\mathbb M_{i,j}$.
Let $\Pi_{i\to j}$ be the family of  all permutations on $V(G)=\{1,2,...,n\}$ taking $i$ to $j$. 
Denote the cycle of $\pi$ permuting $i$ to $j$ by $\pi_{ij}$. For convenience, $\pi_{ij}$ is also
used to denote the set of vertices which corresponds to the elements in the permutation 
cycle $\pi_{ij}$, for example, $V(G)\backslash \pi_{ij}$
denotes the set of vertices in $V(G)$ but not in $\pi_{ij}$. Denote the permutation of $\pi$ restricted on $V(G)\backslash \pi_{ij}$ by  $\pi \backslash \pi_{ij}$. 
Then 
\begin{align*}
\det(\mathbb M_{i, j})
&=\sum_{\pi \in \Pi_{i\to j} } \mbox{sgn}(\pi)\prod_{k\in V(G)\backslash \{i\} } w(k\pi(k))\\
&=\sum_{\pi \in \Pi_{i\to j}}\big (\mbox{sgn}(\pi_{ij}) \prod_{k\in  \pi_{ij} \backslash \{i\}} w(k\pi(k))\big ) \ 
\big (\mbox{sgn}(\pi \backslash \pi_{ij})\prod_{k\in V(G)\backslash \pi_{ij}} w(k\pi(k)) \big )\\
&=\sum_{\pi \in \Pi_{i\to j}}\big( (-1)^{|E(P)|}w(P)\big) \big( w(\mathcal C\cup L) 
w^2(M) (-1)^{|\mathcal C|+|L|+|E(S)|}\big )\\
&=\sum_{P\in \mathcal P_{ij}} w(P) \big(\sum_{S} w(\mathcal C\cup L) w^2(M) 2^{|\mathcal C|}(-1)^{|\mathcal C|+|L|+|E(S)\cup E(P)|}\big ).
\end{align*}
This completes the proof.
\end{proof}

In the above theorem, $\det(\mathbb A_{(ii)})$ can be represented by a formula in terms of all Sachs subgraphs
of the subgraph $(G-i,w)$ obtained from $(G,w)$ by deleting the vertex $i$ and any
incident loop and edges as given 
in Theorem~\ref{thm:weight-det}.
A weighted graph $(G,w)$ is {\em simply invertible} or has a {\em simple inverse} if its inverse is a weighted simple graph (i.e., without loops). The following proposition is a direct corollary of Proposition~\ref{prop-1} and Theorem~\ref{thm:weight-inverse}.

\begin{prop}\label{prop-2}
A weighted graph $(G,w)$ is simply invertible if
and only if for every vertex $i$, $(G-i, w)$ is not invertible.
\end{prop}




The following proposition shows that an invertible simple weighted bipartite graph is always simply
invertible.

\begin{prop}
Every invertible simple weighted bipartite graph is simply invertible. 
\end{prop}
\begin{proof}
Let $(G,w)$ be an invertible weighted graph such that $G$ is simple and bipartite. Then 
$\det(\mathbb A)\ne 0$ where $\mathbb A$ is the adjacency matrix of
$(G,w)$. So $(G,w)$ has at least one Sachs subgraph 
$S=\mathcal C\cup M\cup L$. Further, $L=\emptyset$ because $G$ is simple.
So a simple bipartite graph with a Sachs subgraph has even number 
of vertices. For every vertex $i\in V(G)$, $(G-i,w)$ has 
no Sachs subgraph because $G-i$ is bipartite and has odd number of vertices. 
So $(G-i, w)$ is not invertible.  It follows from Proposition~\ref{prop-2} that 
$(G,w)$ is simply invertible.
\end{proof}

The inverses of simple bipartite graphs with a unique  Sachs subgraph (i.e., a unique perfect matching)
have been discussed in \cite{KM,MM,SC}. In the next section, we consider the inverse of signed graphs with a unique Sachs subgraph and extend results in \cite{KM,MM,SC} for non-bipartite graphs.


\section{Signed graphs with a unique Sachs subgraph}

Let $(G,\sigma)$ be a signed graph. If $(G,\sigma)$ has a unique Sachs 
subgraph, then the determinant of its adjacency matrix is not zero by 
Theorem~\ref{thm:weight-det}. By Proposition~\ref{prop-1}, a signed graph with a unique 
Sachs subgraph is always invertible.
In this section, we study properties of the inverse of signed graphs $(G,\sigma)$. 

Before proceeding to our results, we need some definitions. Let $G$ be a graph and $M$ a matching
$G$. A cycle of $G$ is {\em $M$-alternating} if the edges of $G$ alternate between $M$ and
$E(G)\backslash M$. So an $M$-alternating cycle is always of even size. A path $P$ (or a cycle $C$) of $G$ is
{\em $M$-alternating} if all vertices of $P$ (or $C$) are covered by $M$ and the edges of $P$ (or $C$)
alternate between $M$ and $E(G)\backslash M$. Let $G$ have a Sachs subgraph 
$S=\mathcal C\cup M\cup L$. If $\mathcal C$ contains a cycle of even size, then the edges of the cycle 
can be partitioned into two disjoint matchings. Replacing the cycle in $S$ by any one of these two matchings
generates another Sachs subgraph of $G$. If $S$ is a unique Sachs subgraph of $G$, then every 
cycle in $\mathcal C$ is of odd size. This fact we use repeatedly in the proof of 
the following result which characterizes all simple graphs with a unique Sachs subgraph. 
A family $\mathcal D$ of cycles is {\em independent} if there is no edge joining vertices from
two different cycles in $\mathcal D$. 

\begin{thm}\label{thm:unique-Sachs}
A simple graph $G$ has a unique Sachs subgraph if and only if $G$ can be reduced to 
a family of independent odd cycles by repeatedly deleting pendant edges together with their end-vertices. 
\end{thm}
\begin{proof}
\noindent{\em Sufficiency}: Let $G$ be a simple graph which can be reduced to a family of independent odd cycles by
repeating the deletion of pendent edges together with their end-vertices. Let $M$ be the set of all 
edges of $G$ deleted in the process and $\mathcal C$ be the set of remaining odd cycles. Clearly,
$\mathcal C\cup M$ is a Sachs subgraph. On the other hand, for any Sachs subgraph $S$ of $G$,
$M\subseteq S$. Since $\mathcal C$ is the set of independent odd cycles, $\mathcal C\subseteq S$.
So $S=\mathcal C\cup M$. Hence $G$ has a unique Sachs subgraph.\medskip

\noindent{\em Necessary:} Assume that $G$ has a unique Sachs subgraph $S=\mathcal C\cup M\cup L$. 
Then $L=\emptyset$ because $G$ is simple. 
If $G$ consists of 
a family of independent odd cycles, then we are done. Without loss of generality, we may assume that $G$ is connected  but not an odd cycle. First, we establish the following:\medskip

{\bf Claim:} {\sl $G$ has a pendant edge.}\medskip

\noindent{\em Proof of Claim:} Suppose on the contrary that $G$ does not have a pendant edge.

First, suppose that $\mathcal C=\emptyset$. Then $S$ is a perfect matching of $G$. Choose a longest $S$-alternating
path $P$ (edges of $P$ alternate between $S$ and $E(G)\backslash S$). Let $x$ and $y$ be the 
end-vertices of $P$ and $xx', yy'\in E(P)\cap S$. 
Since $P$ is a longest path and $G$ has no pendant edges, both $x$ and $y$ have respective neighbors $x^+$ and $y^+$ in $V(P)\backslash \{x',y'\}$. Let $C_x:=xx'Px^+x$ be the cycle traveling from $x$ to
$x^+$ through the path $P$ then returning to $x$ through the edge $x^+x$. Similarly, $C_y:=yy'Py^+y$.
If $C_x$ has even size, then the symmetric difference $E(C_x)\oplus S$ is another perfect matching 
of $G$, a contradiction to $G$ having a unique Sachs subgraph. Hence $C_x$ is of odd size. So is $C_y$.

\begin{figure}[!hbtp]\refstepcounter{figure}\label{fig:unique-1}
\begin{center}
\includegraphics[scale=1]{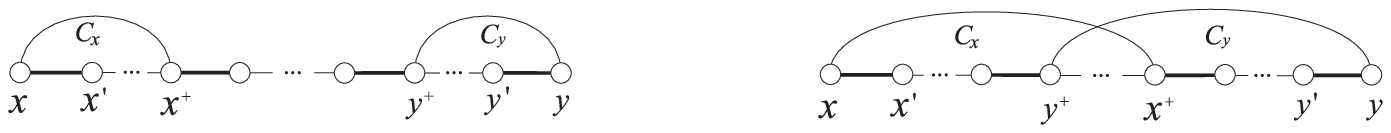}\\
{Figure \ref{fig:unique-1}: The unique Sachs subgraph $S$ is a perfect matching.}
\end{center}
\end{figure}

If $C_x\cap C_y=\emptyset$, then $P-V(C_x\cup C_y)$ has an even number of vertices and 
hence has a perfect matching
$M'$ (see Figure~\ref{fig:unique-1}, Left).  So 
 $(S\backslash E(P)) \cup \{C_x, C_y\}\cup M'$ is another Sachs subgraph of $G$, a contradiction.
So, $C_x\cap C_y\ne \emptyset$ (see Figure~\ref{fig:unique-1}, Right). Let $C:=xPy^+yPx^+x$. 
Since both $C_x$ and $C_y$ have odd sizes, it follows that $C$ is
an $M$-alternating cycle of $G$. Hence the symmetric difference $E(C)\oplus S$ is another perfect matching of $G$,
contradicting that $G$ has a unique Sachs subgraph.

So in the following, assume that $\mathcal C\ne \emptyset$.  Note that every cycle
of $\mathcal C$ has odd sizes. Let $C\in \mathcal C$.  Recall that $S=\mathcal C\cup M$.  Let $x$ be a vertex of $C$ and $P$ be a longest
path starting at $x$ such that the edges of $P$ are alternating between $E(G)\backslash M$ and $M$. 

\begin{figure}[!hbtp]\refstepcounter{figure}\label{fig:unique-2}
\begin{center}
\includegraphics[scale=1.1]{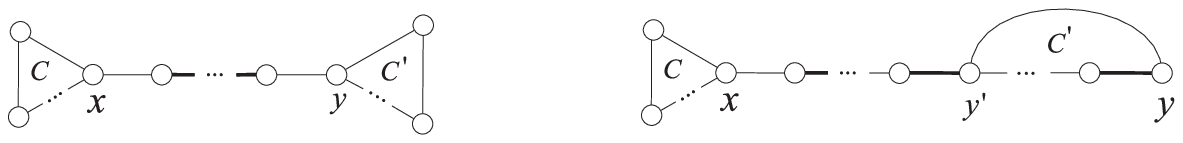}\\
{Figure \ref{fig:unique-2}: The unique Sachs subgraph $S$ contains a cycle $C$.}
\end{center}
\end{figure}

If another end vertex $y$ of $P$ is not covered by the matching $M$, then $y$ is contained in a cycle $C'$ of $\mathcal C$. If  $C'\ne C$, then the graph consisting of $C$, $C'$ and the path $P$ has a perfect matching, denoted by $M'$ (see Figure~\ref{fig:unique-2}, Left). So $G$ has another Sachs subgraph obtained from $S$ by replacing $C$, $C'$ and the edges in $P'\cap M$ by $M'$, a contradiction.
So assume that $C=C'$, then $x,y\in V(C)$. Since $C$ is of odd
size, there is a path $P'$ of $C$ connecting $x$ and $y$ with an odd number of vertices. Then $P'\cup P$ is an odd cycle
and $C-V(P')$ is a path with an even number of vertices. So $G$ has another Sachs subgraph consisting of 
$S\backslash (\{C\}\cup E(P))$, the cycle $P'\cup P$ and a perfect matching of $C-V(P')$, again a contradiction.

So assume that $y$ is covered by the matching $M$ (see Figure~\ref{fig:unique-2}, Right).
Since $G$ has no pendant edges, $y$ has another neighbor $y'$ in $P$. By the same argument as above,
the cycle $C':=y'Pyy'$ consisting of the segment of $P$ from $y'$ to $y$ and the edge $yy'$ is of odd size. 
Let $G'=C\cup P$. Then $G'-V(C')$ has a perfect matching, denoted by $M'$. So $G$ has another Sachs 
subgraph obtained from $S$ by replacing $C$ and $P\cap M$ by $C'$ and $M'$, a contradiction. 
This contradiction completes the proof of our Claim. \qed

\medskip
Now by our Claim, $G$ has a pendant edge $e$. Deleting the edge $e$ together with its end vertices generates 
a subgraph $G_1$ of $G$ which still has a unique Sachs subgraph. Applying the claim on $G_1$, then either $G_1$ 
also contains a pendant edge or else $G_1$ is a family of independent odd cycles. If the former holds, delete
the pendant edge together with its end-vertices. We continue this process until there is no pendant edge 
left. Then the remaining graph consists of independent odd cycles. This completes the proof.
\end{proof}

If $G$ is a bipartite graph with a unique Sachs subgraph, 
then $G$ does not contain odd cycles and hence contains
a pendant edge by Theorem~\ref{thm:unique-Sachs}. So we 
have the following corollary.

\begin{cor}[\cite{SC}, cf.\cite{LP}]
A bipartite graph with a unique perfect matching has a pendant edge.
\end{cor}

For signed graphs with a unique Sachs subgraph, the simply invertible property
implies that the inverse weight function is integral (also call {\em integral inverse}) as 
described in the following theorem, which also generalizes and extends Theorem 2.1 
in \cite{MM} for bipartite graphs with a unique perfect matching.

\begin{thm}\label{thm:unique}
Let $(G,\sigma)$ be a simple signed graph with a unique Sachs subgraph $S$. Then:

$(1)$  $(G, \sigma)$ has an integral inverse if and only if $S$ is a perfect matching;

$(2)$ if $(G,\sigma)$ has a simple inverse, then $S$ is a perfect matching.
\end{thm}
\begin{proof} 
Let $S=\mathcal C\cup M\cup L$ be the unique Sachs subgraph of $(G,\sigma)$. Since
$(G,w)$ is simple, $L=\emptyset$.
Then by Theorem~\ref{thm:weight-det}, it follows that
\[\det(\mathbb A)
=\sigma(\mathcal C)\sigma^2(M)2^{|\mathcal C|} (-1)^{|\mathcal C|+|E(S)|}
=\sigma(\mathcal C)2^{|\mathcal C|} (-1)^{|\mathcal C|+|E(S)|}.\]
Note that every cycle in $\mathcal C$ is of odd size.

(1) First, assume that $S$ is a perfect matching. Then $\mathcal C=\emptyset$ and hence 
$\det(\mathbb A)=(-1)^{|M|}$. 
By Theorem~\ref{thm:weight-inverse},  $(G,\sigma)$ has an integral inverse. 

For the other direction, assume that $(G,\sigma)$ has an integral inverse. Suppose on the contrary 
that $\mathcal C\ne \emptyset$. Let $C$ be a cycle of $\mathcal C$.
 Then $C$ is of odd size. 
Note that, for any vertex $i\in V(C)$, $C$ has a maximum
matching $M_C$ which covers all vertices of $C$ except $i$. By Theorem~\ref{thm:unique-Sachs},
the graph $G-i$, obtained from $G$ by deleting the vertex $i$, still has a unique Sachs subgraph
$S'$. Note that $S$ contains one more cycle (the cycle $C$) than $S'$. By
Theorem~\ref{thm:weight-inverse}, $|(\mathbb A^{-1})_{ii}|=1/2$, which contradicts that 
$(G,\sigma)$ has an integral inverse.  \medskip

(2)  It suffices to show that  $\mathcal C= \emptyset$. If not, let $C\in \mathcal C$. Then
$C$ is an odd cycle. Let $i$ be a vertex of  $C$. By a similar argument as above, we have 
$(\mathbb A^{-1})_{ii}\ne 0$. So the inverse of $(G,\sigma)$ is not simple, a contradiction. 
This completes the proof.
\end{proof}

\noindent{\bf Remark.} The other direction of (2) in the above theorem does not always hold. For example, the graph $G$ in
Figure~\ref{fig:triangle}: $G-w$ is invertible because $G-w$ has a unique Sachs subgraph.

\begin{figure}[!hbtp]\refstepcounter{figure}\label{fig:triangle}
\begin{center}
\includegraphics[scale=.8]{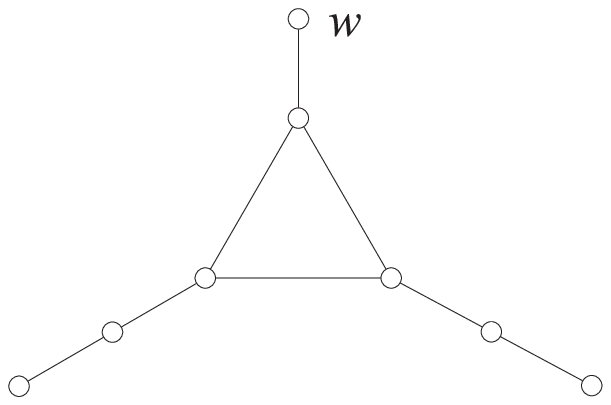}\\
{Figure \ref{fig:triangle}: The graph $G$ with a unique perfect matching but without simple inverse.}
\end{center}
\end{figure}




In the following, we consider two important families of graphs, one is called {\em stellated graphs} \cite{KL, JJS}
and the other is called {\em corona graphs} \cite{MM}. Let $G$ be a graph. 
The {\em stellated graph} of $G$, denoted by $\mbox{st}(G)$, is the line graph of the subdivision of $G$ obtained from $G$ by subdividing every edge once (see Figure~\ref{fig:stellated}). The stellated graph of $G$ is also called
inflated graph  \cite{OF98} or para-line graph of $G$ \cite{TS}. A graph $G$ is called a {\em stellated} graph if $G=\mbox{st}(H)$ for some
graph $H$. 
The spectrum of stellated graphs (or para-line graphs) have been studied in \cite{TS}. For chemical applications of the stellated graphs of trees, refer to \cite{KL}.

\begin{figure}[!hbtp]\refstepcounter{figure}\label{fig:stellated}
\begin{center}
\includegraphics[scale=1]{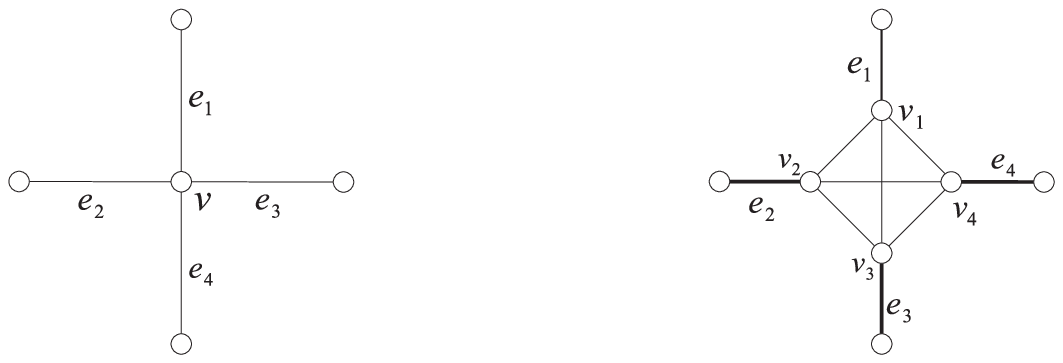}\\
{Figure \ref{fig:triangle}: The graph $K_{1,4}$ (left) and its stellated graph $\mbox{st}(K_{1,4})$ (right).}
\end{center}
\end{figure}

\begin{lemma}\label{lem:stellated}
Let $G$ be a connected simple graph with at least one edge. Then its stellated graph $\mbox{st}(G)$ has a perfect matching, and $\mbox{st}(G)$ has a unique perfect matching if and only if $G$ is a tree.
\end{lemma}
\begin{proof}
Let $G$ be a connected simple graph with at least one edge, and $\mbox{st}(G)$ the stellated graph of $G$.
For a vertex $v$ of $G$, assume its degree is $d(v)$. Let $E_v=\{e_1,e_2,...,e_{d(v)}\}$ be the set of edges incident with
$v$. 
Then 
$\mbox{st}(G)$ can be treated as a graph obtained
from $G$ by replacing every vertex $v$ by a clique consisting of $v_1, v_2, ..., v_{d(v)}$ such that
for an edge $uv\in E(G)$, joining $v_i$ and $u_j$ if $uv=e_i\in E_v$ and $uv=e'_j\in E_u$. 
So $M=\{v_iu_j| v_iu_j\in E(\mbox{st}(G))\mbox{ and }uv\in E(G)\}$ is a perfect matching of $\mbox{st}(G)$ (for example thick edges in $\mbox{st}(K_{1,4})$ in Figure~\ref{fig:stellated}).

On the other hand, $G$ is obtained from $\mbox{st}(G)$ by contracting these maximal clique to a 
vertex of $G$. Let $M$ be a perfect matching of $\mbox{st}(G)$. Note that an $M$-alternating 
cycle of $\mbox{st}(G)$  corresponds to a cycle of $G$. So $\mbox{st}(G)$
has a unique perfect matching $M$ if and only if $\mbox{st}(G)$ has no $M$-alternating cycles. 
Hence $G$ has no cycles. So $G$ is a tree.
\end{proof}
 
By Lemma~\ref{lem:stellated}, a stellated graph without isolated vertices always has a perfect matching.
So if a stellated graph has a unique Sachs subgraph, then it has a unique perfect matching. On
the other hand, if a stellated graph has a unique perfect matching, then it is a stellated graph of 
a tree and hence has a unique Sachs subgraph.
Let $T$ be a tree. For any two vertices of $T$, there is exactly one path joining them.
For any two vertices $i$ and $j$ of $\mbox{st}(T)$, there is at most one 
$M$-alternating path $P$ joining $i$
and $j$ because $P/(E(P)\backslash M)$ is a path in $T$. If $i$ and $j$ are joined 
by an $M$-alternating path $P_{ij}$, let $\tau(i,j):=|E(P_{ij})\backslash M|$. 
Now, we proceed to consider the inverse of signed stellated graphs with a unique Sachs subgraph
(equivalently, a unique perfect matching).

\begin{figure}[!hbtp]\refstepcounter{figure}\label{fig:inverse-stellated}
\begin{center}
\includegraphics[scale=1]{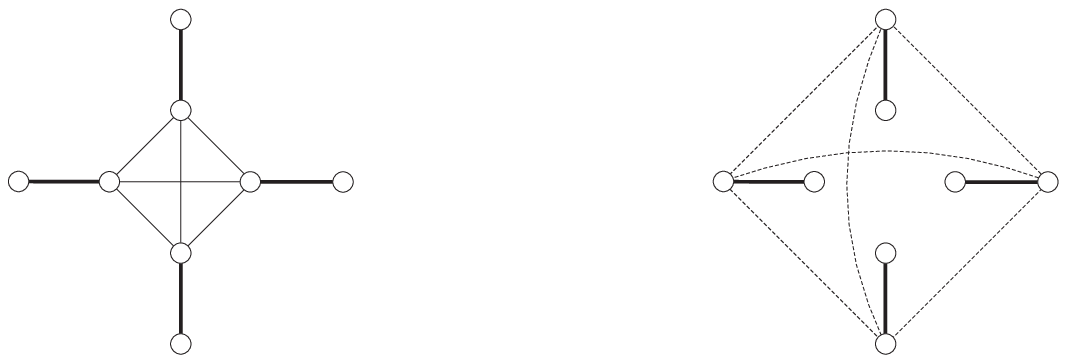}\\
{Figure \ref{fig:inverse-stellated}: The stellated graph $K_{1,4}$ (left) and its inverse (right): dashed edges
has weight $-1$, solid edges have weight 1, and thick edges form a perfect matching.}
\end{center}
\end{figure}

\begin{thm}\label{thm:stellated-inverse}
Let $(G,\sigma)$ be a signed stellated graph with a unique perfect matching $M$. Then $(G,\sigma)$ 
has an inverse $(G^{-1}, \sigma^{-1})$ which is a signed graph such that, for any two vertices $i$ and $j$, 

(1) $ij\in E(G^{-1})$ if and only if 
there is an $M$-alternating path $P_{ij}$ joining $i$ and $j$ in $(G,\sigma)$; and

(2) $\sigma^{-1}(ij)=(-1)^{\tau(i,j)}\sigma(P_{ij})$.
\end{thm} 
\begin{proof}
Let $(G,\sigma)$ be a signed stellated graph with a unique perfect matching. By the above argument,
$M$ is the unique Sachs subgraph of $(G,\sigma)$. 
Let $\mathbb A$ be the adjacency matrix of $(G,\sigma)$. Then by Corollary~\ref{cor-sign}, $\det(\mathbb A)=(-1)^{|M|}$. Hence $(G,\sigma)$ has inverse $(G^{-1},\sigma^{-1})$.

By Lemma~\ref{lem:stellated}, $G$ is the stellated graph of a tree $T$. 
So every cycle $C$ of $G$ is contained in a maximal clique corresponding to a vertex of $T$.
Note that $G-V(C)$ has at least $|C|$ components each without a perfect matching. 
Hence, for any vertex $i$ of  $G$, a Sachs subgraph of $G-i$ contains no 
cycles and hence is a perfect matching. 
However, $G-i$ has an odd number of vertices. Hence $G-i$ has no Sachs subgraph.
So  $(G^{-1},\sigma^{-1})$ is simple.

Let $P$ be a path joining two vertices $i$ and $j$. Then $G-V(P)$ has a
Sachs subgraph if and only if $P$ is an $M$-alternating path.  
By Theorem~\ref{thm:weight-inverse}, 
$(\mathbb A^{-1})_{ij}\ne 0$ if and only if there is an $M$-alternating path joining $i$ and $j$. 
It follows immediately that $i$ and $j$ are adjacent in $G^{-1}$ if
and only if there is an $M$-alternating path $P_{ij}$ joining them.

Note that, for any two vertices $i$ and $j$ of $G$, there is 
at most one $M$-alternating path $P_{ij}$ joining $i$ and $j$. 
So if $ij\in E(G^{-1})$, by Theorem~\ref{thm:weight-inverse},
$\sigma^{-1}(ij)=(\mathbb A^{-1})_{ij}=\sigma(P_{ij}) (-1)^{\tau(i,j)}$. This completes the proof.
\end{proof}

The {\em corona} of a graph $H$ is a graph $G$ obtained from $H$ by 
adding a neighbor of degree 1 to each vertex
of $H$ (\cite{SC}). A graph $G$ is a {\em corona graph} if it is the corona of some graph.
A corona graph has an even number of vertices and half of them have degree 1. So
a corona graph has a unique Sachs subgraph that is a perfect matching.  
The inverse of a bipartite
corona graph has been discussed in \cite{BG, MM,SC}.
A weighted graph $(G,w)$ is {\em self-invertible} if it has an inverse $(G^{-1},w^{-1})$ 
such that $G^{-1}$ is isomorphic to $G$. Note that, their
weight-functions $w$ and $w^{-1}$ may be different.
Based on Godsil's definition of inverse, Simoin and Cao \cite{SC} show that if 
$G$ is a bipartite graph with a unique perfect matching $M$ such that $G/M$ is 
bipartite, then $G$ is self-invertible if and only if $G$ is isomorphic to a 
bipartite corona graph (see also~\cite{BG}). The self-invertibility of a corona bipartite graphs  
is also verified for McLeman-McNicholas's definition \cite{MM}.
This result is partially generalized to non-bipartite signed graphs as follows.

\begin{thm}\label{thm:corona}
Let $(G,\sigma)$ be a simple signed corona graph.
Then $(G,\sigma)$ is self-invertible.
\end{thm}
\begin{proof} 
Let $V(G)=\{v_1,v_2,...,v_n, u_1,u_2,...,u_n\}$ such that $d_G(u_i)=1$ and $v_iu_i\in E(G)$.
Then $G$ has a unique Sachs subgraph which is a perfect matching $M=\{v_iu_i| 0\le i\le n\}$.
Let $\mathbb A$ be the adjacency matrix of $(G,\sigma)$. 
Then $\det(\mathbb A)=(-1)^{|M|}$ by Corollary~\ref{cor-sign}. So $(G,\sigma)$ has inverse $(G^{-1},\sigma^{-1})$.

Let $P_{xy}$ be a path joining two vertices $x$ and $y$ of $G$. Since $G$ is a simple corona graph, $G-V(P_{xy})$ has a Sachs subgraph if and only if, for 
each $v_i\in V(P_{xy})$, we have $u_i\in V(P_{xy})$.  It follows immediately that $xy=u_iv_i$, or $x=u_i$, $y=u_j$ and $P_{xy}=u_iv_iv_ju_j$ for some $i$ and $j$.
By
Theorem~\ref{thm:weight-inverse}, for any two vertices $x, y\in V(G^{-1})$, it follows that $\sigma^{-1}(xy)=(\mathbb 
A^{-1})_{xy}=\sigma(P_{xy})(-1)^{\tau(x,y)}$. Hence,
\[
\sigma^{-1}(xy)=\left\{
 \begin{array}{lll}
\sigma(xy)  &\mbox{if $xy=u_iv_i$;}\\
-\sigma(P_{u_iu_j}) &\mbox{if $xy=u_iu_j$ where $P_{u_iu_j}=u_iv_iv_ju_j$;}\\
0  &\mbox{otherwise.}
 \end{array}
 \right.
\]
So $(G^{-1},\sigma^{-1})$ is a simple signed graph. In $G^{-1}$, $d_{G^{-1}}(v_i)=1$, $v_iu_i\in E(G)$, and $u_iu_j\in E(G^{-1})$
if and only if $v_iv_j\in E(G)$. So the mapping $\phi:V(G) \to V(G^{-1})$  
that $\phi(u_i)=v_i$ and $\phi(v_i)=u_i$ is an isomorphism between $G$ and $G^{-1}$. Hence $(G,\sigma)$ is self-invertible.
\end{proof}

\section{Eigenvalues}

Let $(G,w)$ be a weighted graph with $n$ vertices. Assume that $\lambda_1(G,w)\ge \lambda_2(G,w)\ge \cdots \lambda_n(G,w)$
are all eigenvalues of $(G,w)$. 
 The {\em spectrum} of $(G,w)$ is the family of all eigenvalues. 
The spectrum of $(G,w)$ {\em splits} about the origin if it has half positive eigenvalues and half negative eigenvalues. Let $H=\lfloor (n+1)/2\rfloor$ and $L=\lceil (n+1)/2\rceil$.
The median eigenvalues are $\lambda_{H}(G,w)$ and
$\lambda_L(G,w)$. If the spectrum of $(G, w)$ splits about the origin, then 
$\lambda_H(G, w)> 0> \lambda_L(G,w)$. Many graphs representing stable molecules 
have a spectrum split
about the origin. 
In \cite{KL}, it has been shown that the spectrum of the stellated graph of a tree splits about the origin. 

\begin{prop}\label{prop}
Let $(G,w)$ be an invertible graph and $(G^{-1},w^{-1})$ be its inverse. If the spectrum of $(G,w)$ splits about the origin,
so does the spectrum of $(G^{-1},w^{-1})$ and $\lambda_{H}(G,w)=1/\lambda_1(G^{-1},w^{-1})$ and $\lambda_{L}(G,w)=1/\lambda_n(G^{-1},w^{-1})$.
\end{prop}
\begin{proof}
Since $(G,w)$ is invertible, its adjacency matrix $\mathbb A$ has inverse $\mathbb A^{-1}$
which is the adjacency matrix of $(G^{-1},w^{-1})$. Note that, $\lambda$ is an eigenvalue of
$\mathbb A$ if and only if $1/\lambda$ is an eigenvalue of $\mathbb A^{-1}$. 
Since the spectrum of $(G,w)$ splits about the origin, then the proposition follows.
\end{proof}

For non-weighted graphs $G$, it is well-known that  $G$ is bipartite if and only if the spectrum of $G$ is symmetric about
the origin. But for a weighted graph, one direction is necessarily 
but not the other. 

\begin{prop}\label{prop-3}
Let $(G,w)$ be a weighted simple bipartite graph. Then the spectrum of $(G,w)$ is symmetric 
about the origin. 
\end{prop}
\begin{proof}
Since $(G,w)$ is a weighted simple bipartite graph, then its adjacency matrix
$\mathbb A=\begin{bmatrix} \mathbf 0 & B\\
                                                     B^{\intercal} & \mathbf 0\end{bmatrix}$. 
Let $\lambda$ be an eigenvalue of $\mathbb A $, and 
$\mathbf x=\begin{bmatrix} \mathbf u \\ \mathbf v \end{bmatrix}$ an eigenvector of $\lambda$.
It is easily seen that $-\lambda$ is an eigenvalue of $\mathbb A $ with eigenvector
$\begin{bmatrix} \mathbf u \\ -\mathbf v \end{bmatrix}$. 
\end{proof}

\begin{figure}[!hbtp]\refstepcounter{figure}\label{fig:2tri}
\begin{center}
\includegraphics[scale=.8]{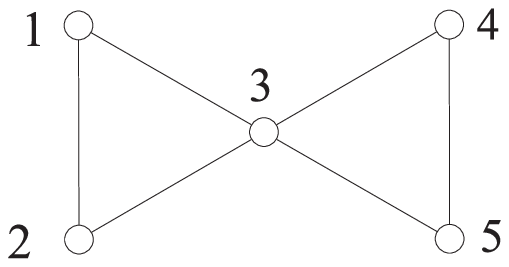}\\
{Figure \ref{fig:2tri}: A non-bipartite weighted graph $(G, w)$  which has spectrum symmetric about the origin.}
\end{center}
\end{figure}

\noindent {\bf Remark.} The other direction of the above proposition holds if $w$ is a 
positive function. If $w$ is not positive, then it is not always true. For example, the weighted 
graph $(G,w)$ in Figure~\ref{fig:2tri} with $w(e)=1$ if $e$ is an edge of the triangle $1231$ and $-1$ otherwise: if $\lambda$ is an eigenvalue of $(G,w)$ with eigenvector $\mathbf x=[x_1, x_2, x_3, x_4,x_5]^\intercal$, then $-\lambda$ is also an eigenvalue of $(G,w)$ with eigenvector $\mathbf x'=[x_5,x_4,x_3,x_2,x_1]^\intercal$.

\begin{thm}
Let $(G,\sigma)$ be a signed graph with a perfect matching $M$. If 
for any weight-function $w: E(G)\to I=[-1,1]$ such that $w(e)\in \{-1,1\}$ for every
$e\in M$, $(G,w)$ is invertible, then the spectrum of $(G,\sigma)$
splits about the origin.
\end{thm}
\begin{proof}
Let $(G, w)$ be a weighted graph and $w: E(G)\to [-1,1]$ be a weight-function
such that $w(e)\in \{-1,1\}$ for $e\in M$. Let $\mathbb A$ be the adjacency
matrix of $(G,w)$. 

Since $(G,w)$ is invertible, so is $\mathbb A$. Hence $\det(\mathbb A)\ne 0$. 
Let $\mathbb A_0$ be the adjacency matrix of $(G,w_0)$ such that $w_0(e)=0$ if $e\in E(G)\backslash M$
and $w_0(e)\in \{-1,1\}$ if $e\in M$. Then $\mathbb A_0$ is the adjacency matrix of $M$, 
an invertible bipartite graph, whose spectrum is symmetric about the origin and not equal to 0 by 
Proposition~\ref{prop-3}.

Let $\lambda (w)$ be an eigenvalue of $\mathbb A$ and $\mathbf x$ be an eigenvector of $\lambda (w)$.
Then 
\[\lambda (w)=\frac{\langle\mathbb A \mathbf x, \mathbf x\rangle}{\langle\mathbf x,\mathbf x\rangle}.
\]
Hence $\lambda (w)$ is a function of $w$, which is continuous on $[-1,1]$. 
Since $\det(\mathbb A)\ne 0$, it follows that $\langle\mathbb A \mathbf x, \mathbf x\rangle\ne 0$. 
Hence $\lambda (w)\ne 0$. So for any $w: E(G)\to [-1,1]$, $\lambda(w)$ 
has the same sign. Let $\sigma(e)\to \{-1,1\}$ for any $e\in E(G)$. Then $\lambda(\sigma) \lambda(w_0)>0$. Hence the spectrum of $(G,\sigma)$ splits about the origin.
\end{proof}


\begin{thm}\label{thm:split}
Let $(G,\sigma)$ be a signed graph with a unique Sachs subgraph. If $G$ has a  
perfect matching $M$,
then the spectrum of $(G,\sigma)$ splits about the origin. 
\end{thm}
\begin{proof}
Consider $(G,w)$ where $w: E(G)\to [-1,1]$ such that $w:M\to \{-1,1\}$. 
Since $G$ has a unique Sachs subgraph which is a perfect matching $M$, by Corollary~\ref{cor-sign}, we have 
\[\det(\mathbb A)=(-1)^{|M|}.\] 
Hence $(G,w)$ is invertible.
By the above theorem,
the spectrum of $(G,w)$ splits about the origin. So does the spectrum of $(G,\sigma)$. 
\end{proof}

Since both a stellated graph without isolated vertices and a corona graph 
have a perfect matching, the following results are direct corollaries of Theorem~\ref{thm:split}.

\begin{cor}\label{cor-stellated}
Let $G$ be a stellated graph with a unique Sachs subgraph. Then the spectrum of $G$ splits about the
origin if $G$ does not contain an isolated vertex.
\end{cor}

\begin{cor}\label{cor-corona}
The spectrum of a corona graph splits about the origin.
\end{cor}

\section{Median eigenvalues}

Let $G$ be a graph and $\lambda_1\ge \lambda_2\ge \cdots \ge \lambda_n$, the eigenvalues of
$G$.
The difference of median eigenvalues 
$\Delta=\lambda_{H}-\lambda_{L}$ in chemistry
is called the HOMO-LUMO gap of the neutral pi-network molecule corresponding
the graph $G$, so that $\Delta$ is also
called the {\em HOMO-LUMO gap} of $G$ \cite{FP}. 
A graph $G$ is a special signed graph $(G,\sigma)$ such
that $\sigma: E(G)\to \{1\}$. So all results above can be applied to graphs as special cases.

Let $\mathbb A$ be the adjacency matrix of $G$. 
Define $x: V(G)\to \mathbb R$ such that
$x(i)=x_i$, and let $\mathbf x=(x_1, x_2, ..., x_n)^{\intercal}\in \mathbb R^n$ where.
Then by Rayleigh-Ritz quotient,
\[\lambda_1(G,\sigma)=\max_{\mathbf x \in \mathbb R^n}\frac{\langle\mathbb A \mathbf x, \mathbf x\rangle}{\langle\mathbf x, \mathbf x\rangle}=\max_{\mathbf x\in \mathbb R^n}\frac{\sum (\mathbb A)_{ij}x_ix_j}{\parallel x\parallel}\]
and 
\[\lambda_n(G,\sigma)=\min_{\mathbf x\in \mathbb R^n}\frac{\langle\mathbb A \mathbf x, \mathbf x\rangle}{\langle\mathbf x, \mathbf x\rangle}=\min_{\mathbf x\in \mathbb R^n}\frac{\sum (\mathbb A)_{ij}x_ix_j}{\parallel x\parallel}.\]\medskip



\begin{thm}\label{thm:HL-Stellated}
Let $G$ be a stellated graph of a tree with at least two vertices. Then $-1\le \lambda_L(G)<0<\lambda_H(G)\le 1$.
\end{thm}
\begin{proof} Let $G$ be a stellated graph of a tree $T$.
By Lemma~\ref{lem:stellated}, then $G$ has a unique perfect
matching $M$. Hence $G$ has an inverse which is a signed graph $(G^{-1},\sigma)$ by Theorem~\ref{thm:stellated-inverse}. By Corollary~\ref{cor-stellated}, we have 
$\lambda_L(G)<0<\lambda_H(G)$. 
In order to show $\lambda_H(G)\le 1$ and $\lambda_L(G)>-1$, it suffices to show that 
$\lambda_1(G^{-1}, \sigma )\ge 1$ and $\lambda_n(G^{-1},\sigma )<-1$
by Proposition~\ref{prop}.

Assume that $T$ has $k$ leaves 
with degree sequences as $1=d_1=\cdots =d_k<  d_{k+1}\le \cdots  d_{t-1}\le d_t$. Then in $G$, each vertex $v_{\ell}$
of $T$
with degree $d_{\ell}$  is replaced by a clique with size $d_{\ell}$, denoted by $K_{d_{\ell}}$. 

Let $i$ and $j$ be two vertices of $G$. 
By Theorem~\ref{thm:stellated-inverse},  $ij\in E(G^{-1})$ if
and only if there is an $M$-alternating path $P_{ij}$ joining them, and $\sigma(ij)=(-1)^{|E(P_{ij})\backslash M|}$. If $ij\in M$, then $\sigma(ij)=1$ that implies that $M$ is also a perfect matching of $G^{-1}$. 
If $i$ and $j$ satisfies that $ii', jj'\in M$ and $i'j'\in E(K_{d_{\ell}})$ for some $\ell$, then $i'ijj'$ is the only one $M$-alternating 
path of $G$ joining $i'$ and $j'$. Hence $ij\in E(G^{-1})$ and 
$\sigma(ij)=(-1)^{|E(P)\backslash M|}=-1$. Hence, $G^{-1}$ has a clique $K_{d_{\ell}}'$ corresponding to $K_{d_{\ell}}$
consisting of vertices in $N(V(K_{d_i}))$ and every edge in  $K_{d_i}'$ has weight $-1$.
For example, see Figure~\ref{fig:inverse}: two cliques of order 4 illustrated in dashed lines.

\begin{figure}[!hbtp]\refstepcounter{figure}\label{fig:inverse}
\begin{center}
\includegraphics[scale=1.3]{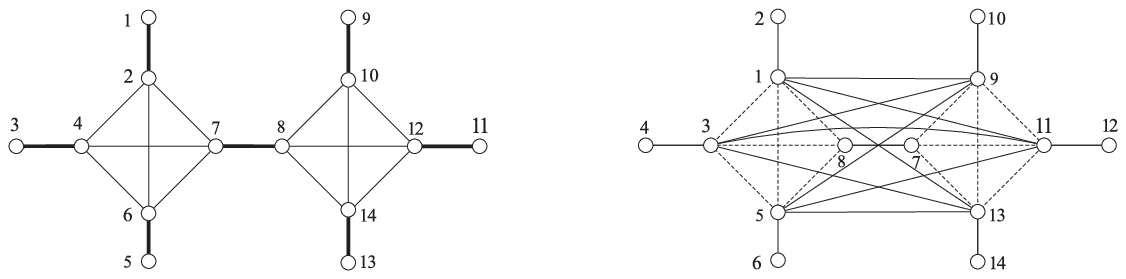}\\
{Figure \ref{fig:inverse}: A stellated graph of a tree (left: thick lines form a perfect matching) and its inverse (right: dashed lines have weight $-1$ and others have weight 1).}
\end{center}
\end{figure}

First, we show that $\lambda_1(G^{-1}, \sigma)\ge 1$. Let $Q$ be the graph obtained from 
$(G^{-1},\sigma)$
by contradicting all edges of $M$ which have positive signature.  
For any ordering of the vertices of $Q$: $q_1, q_2,\cdots, q_l$, let 
$E(q_{\alpha}):=\{q_{\gamma}q_{\alpha}| q_{\gamma}q_{\alpha}\in E(Q)\mbox{ and } \gamma<\alpha\}$.
Assign a weight $x(q_{\alpha})\in \{-1,1\}$ to each vertex $q_{\alpha}$ such that
\[\sum_{q_{\gamma}q_{\alpha} \in E(q_{\alpha})} x(q_{\gamma})\sigma(q_{\gamma}q_{\alpha})x(q_{\alpha})\ge 0.\]
The weight-function always exists because we can adjust the sign of $x(q_{\alpha})$ to change
the above inequality. Hence, $Q$ has a weight-function $x: V(Q)\to \{-1,1\}$ such that
\[\sum_{q_{\gamma} q_{\alpha}\in E(Q)} x(q_{\gamma})\sigma(q_{\gamma}q_{\alpha})x(q_{\alpha})\ge 0.\]
Let $i, j\in V(G^{-1})$ and $ij\in M$. Assume that  $ij$ is contracted to a vertex $q_{\alpha}$ 
in $Q$.
Now, extend the weight-function $x$ to $V(G^{-1})$ such that $x(i)=x(j)=x(q_{\alpha})$, and
define the vector $\mathbf x: V(G^{-1})\to \mathbb \{-1,1\}^{n}$ such that $x_i =x(i)$.  Let $\mathbb A$ be the
adjacency matrix of $(G^{-1},\sigma)$. Note that $E(G^{-1})=E(Q) \cup M$. 
Then it follows that 
\begin{align*}\langle\mathbb A \mathbf x, \mathbf x\rangle=2\sum_{ij\in E(G^{-1})} (\mathbb A)_{ij}x_ix_j
&=2\sum_{ij\in E(Q)} (\mathbb A)_{ij}x_ix_j +2\sum_{ij\in M} (\mathbb A)_{ij}x_ix_j\\
&=2\sum_{ij\in E(Q)} x_i\sigma(ij)x_j +2\sum_{ij\in M} x_i\sigma(ij)x_j\\
&\ge 2 |M|.
\end{align*}
Further, 
\[\lambda_1(G^{-1},\sigma)\ge 
\frac{2\sum_{ij\in E(G^{-1})} (\mathbb A)_{ij}x_ix_j}{\parallel x\parallel}\ge \frac{2|M|}{|V(G)|}=1.\]

In the following, we concentrate on the upper bound for $\lambda_n(G^{-1}, \sigma)$. Let $E_P$ 
be the set of all pendant edges (incident with a degree-1 vertex). Then $E_P\subseteq M$. So 
for any edge $ij\in E_P$, $\sigma(ij)=1$. Let $R$ be
the graph obtained from $G^{-1}$ by contracting all edges in $E_P$ and all cliques $K_{d_{\ell}}'$ for $k+1\le \ell\le n$.
By a similar argument as above, $R$ has a weight-function $x: V(R)\to \{-1,1\}$ such 
that
\[\sum_{r_{\gamma}r_{\alpha}\in E(R)} x(r_{\gamma})\sigma(r_{\gamma}r_{\alpha})x(r_{\alpha})\le 0,\] 
where $r_{\gamma}, r_{\alpha}$ are vertices of $R$. Now, define the vector $\mathbf x$
such that $x_i=x(r_{\alpha})$ if $i\in V(K_{\alpha}')$, and then $x_j=-x_i$ if $ij\in E_P$ and
$i\in V(K_{\alpha}')$. Note that $E(G^{-1})=E(R)\cup \big (\bigcup E(K_{\alpha}')\big ) \cup E_P$. So 
\begin{align}
\langle\mathbb A \mathbf x, \mathbf x\rangle &=2\sum_{ij\in E(G^{-1})} (\mathbb A)_{ij}x_ix_j\\
&=2\sum_{ij\in E(R)} \sigma(ij)x_ix_j +2\sum_{d_{k+1}\le \alpha\le d_t}\sum_{ij\in E(K'_{\alpha})} \sigma( ij)x_ix_j+2\sum_{ij\in E_P} \sigma(ij)x_ix_j\\
&\le 2\sum_{d_{k+1}\le \alpha\le d_t}\sum_{ij\in E(K'_{\alpha})} (-1)-2|E_P|\\
&= -2\sum_{k+1\le \ell\le t} {d_{\ell}\choose 2}-2|E_P|.
\end{align}
Hence, 
\begin{align}
\lambda_n(G^{-1},\sigma)\le 
\frac{2\sum_{ij\in E(G^{-1})} (\mathbb A)_{ij}x_ix_j}{\parallel x\parallel}\le \frac{-2\sum_{k+1\le \ell\le t} {d_{\ell}\choose 2}-2|E_P|}{|V(G)|}.
\end{align} 
Since a vertex of $G^{-1}$ is either a pendent vertex or contained in $K_{\alpha}'$ for some $\alpha$,
it follows that $2\sum_{k\le \ell\le t} {d_{\ell}\choose 2}+2|E_P|\ge |V(G^{-1})|$ and equality holds if 
and only if $G$ is $K_2$. 
So $\lambda_n(G^{-1},\sigma)\le -1$.
This completes the proof.
\end{proof}

\noindent{\bf Remark.} For some specific trees, a better bound for HOMO-LUMO gap of 
their stellated graphs could be obtained
 by the method used in the above proof. For example, the alkanes, trees with only degree-1
vertices and degree-4 vertices. If $G$ is a stellated graph of an alkane, then 
$d_{k+1}=\cdots =d_{t}=4$ and hence ${d_{\ell} \choose 2}= 6$. On the other hand, by
degree condition, it is easily deduced that $k=(2|V(G)|-2)/3$ and $t-k=(|V(G)|+2)/3$.
So, by Inequality (5), $\lambda_n(G^{-1},\sigma)\le -2(6(t-k)+2k)/|V(G)|<-10/3$. 
Hence the HOMO-LUMO gap for
stellated graphs of alkanes is at most 1.3.

\begin{thm}
Let $G$ be a connected corona graph. Then $-1\le \lambda_L(G)<0< \lambda_H(G)\le 1$.
\end{thm}
\begin{proof}
Let $G$ be the corona of a connected graph $H$. Then $G$ has a unique Sachs subgraph which is a perfect matching $M$. By Theorem~\ref{thm:corona}, $G$ has an inverse $(G^{-1},\sigma)$ and 
$G$ is isomorphic to $G^{-1}$. 
For an edge $ij\in E(G^{-1})$, $\sigma(ij)=-1$ if $ij\notin M$ and $\sigma(ij)=1$ if $ij\in M$.
Since $G^{-1}$ is also a corona graph, every edge $ij\in M$ is incident with a vertex of degree-1. 

Let $\mathbb A$ be the adjacency matrix of $(G^{-1},\sigma)$. A similar argument as the proof of 
Theorem~\ref{thm:HL-Stellated} shows that $\lambda_1(G^{-1},\sigma)\ge 1$. 

For the upper bound of $\lambda_n(G^{-1},\sigma)$, let $\mathbf x$ be the vector such 
that $x_i=1$ if $i\in V(H)$ and $x_i=-1$ if $i\in V(G)\backslash V(H)$. Note that all vertices in
$V(G)\backslash V(H)$ has degree 1. 
Hence 
\begin{align}
\lambda_n(G^{-1}, \sigma)\le\frac{\langle \mathbb A \mathbf x, \mathbf x\rangle}{\parallel \mathbf x\parallel}=\frac{2\sum_{ij}\sigma(ij)x_ix_j}{|V(G)|}=\frac{-2|E(G)|}{|V(G)|}\le -1.
\end{align}
The last inequality in (6) holds if and only if $G$ is a $K_2$. (Otherwise it will be $-3/2$.)

By  Proposition~\ref{prop} and Corollary~\ref{cor-corona}, we have $\lambda_H(G)=1/\lambda_1(G^{-1},\sigma)\le 1$ and 
$\lambda_L(G)=1/\lambda(G^{-1},\sigma)\ge -1$. This completes the proof. 
\end{proof}

\section*{Acknowledgement}

The research was supported by the Welch Foundation of 
Houston, Texas (grant BD-0894).  Some of the research was conducted while
the first author was visiting
Texas A\&M University at Galveston.

\end{document}